\documentclass[10pt]{amsart}
\usepackage{amsmath, amsthm, amssymb}
\usepackage{graphics}
\usepackage{pinlabel}
\usepackage{color}
\usepackage[colorlinks,citecolor=blue,linkcolor=red]{hyperref}
\usepackage{graphicx}
\usepackage{nicefrac}						
\usepackage{mathtools}
\usepackage{float}
\usepackage{tikz-cd}
\usepackage{enumerate}

\usepackage{secdot}
\usepackage{caption}

\usepackage[margin=1.1in]{geometry}

\usepackage{colonequals}
\allowdisplaybreaks

\newtheorem{mydef}{Definition}
\newtheorem{thm}{Theorem}
\newtheorem{cor}{Corollary}
\newtheorem*{mainthm*}{Main Theorem}
\newtheorem{conj}{Question}
\newtheorem*{conj1*}{Question}
\newtheorem{prop}{Proposition}
\newtheorem{lemma}{Lemma}

\DeclareMathOperator{\cyl}{cyl}
\DeclareMathOperator{\CAT}{CAT}

\numberwithin{thm}{section} 
\numberwithin{prop}{section} 
\numberwithin{lemma}{section} 
\numberwithin{mydef}{section}

\title{Length spectra of flat metrics coming from $q$-differentials}
\author{Marissa Loving}					
\date{}

\begin{document}

\maketitle 

\begin{abstract} When geometric structures on surfaces are determined by the lengths of curves, it is natural to ask: {\em which} curves' lengths do we really need to know? It is a result of Duchin--Leininger--Rafi that any flat metric induced by a unit-norm quadratic differential is determined by its marked simple length spectrum. We generalize the notion of simple curves to that of {\em $q$-simple curves}, for any positive integer $q$, and show that the lengths of $q$-simple curves suffice to determine a non-positively curved Euclidean cone metric induced by a $q$-differential.\end{abstract}

\section{Introduction}

Let $S$ be a closed surface and let $\mathcal C = \mathcal{C}(S)$ be the set of homotopy classes of closed curves on $S$. For an isotopy class of metrics $m$ and $\alpha \in \mathcal C$, we denote the infimum of lengths of representatives of $\alpha$ in a representative metric for $m$ by $\ell_{m}(\alpha)$. We define the {\bf marked $\Sigma$-length spectrum} of $m$ for a subset $\Sigma \subset \mathcal C$ to be: \[\lambda_{\Sigma}(m) = ( \ell_{m}(\alpha))_{\alpha \in \Sigma} \in \mathbb R^{\Sigma}.\] 
 
Let $\mathcal G = \mathcal G(S)$ be a family of metrics, up to isotopy. If every $m \in \mathcal G$ is determined by $\lambda_{\Sigma}(m)$ for $\Sigma \subset \mathcal C$, then we say that $\Sigma$ is {\bf spectrally rigid} over $\mathcal G$. It is known that $\mathcal C$ is spectrally rigid for the set of all non-positively curved (NPC) Riemannian cone metrics, see \cite{otal90, croke90, crokefathifeld92, herspaul97, banklein18, constantine15}. In the presence of some ``local symmetry" for the metrics a much smaller set of curves may suffice. For example, a classical result of Fricke shows that the set $\mathcal S = \mathcal S(S)$ of simple closed curves is spectrally rigid for the Teichm\"uller space of hyperbolic metrics. Duchin--Leininger--Rafi \cite{duchleinrafi10} similarly proved that $\mathcal S$ is spectrally rigid for the space $\text{Flat}_2(S)$ of NPC Euclidean cone metrics coming from unit area holomorphic quadratic differentials.

Here we consider the space $\text{Flat}_q(S)$ of unit area flat metrics coming from holomorphic $q$-differentials for any integer $q \geq 1$. We define a new class of curves, denoted $\mathcal{S}_q$, called the {\em $q$-simple} curves on $S$ (see Definition \ref{def:q-simple}), and prove the following.  

\begin{mainthm*} Let $S$ be a closed surface of genus $\geq 2$. For any $q \geq 1$, $\mathcal{S}_q$ is spectrally rigid over $\text{Flat}_q(S)$. \end{mainthm*}

\subsection*{Acknowledgment} The author was supported by an NSF Graduate Research Fellowship under Grant No. DGE 1144245, as well as NSF grants DMS 1107452, 1107263, 1107367 ``RNMS: GEometric structures And Representation varieties", which allowed her to begin this project during the 2017-2018 Warwick EPSRC Symposium on Geometry, Topology and Dynamics in Low Dimensions. The author would like to thank Justin Lanier and Sunny Yang Xiao for helpful conversations that improved the exposition of this paper and Ser-Wei Fu for pointing out an error in an earlier version of the proof of Proposition \ref{prop:character} in the case of even $q$. Most importantly, the author is grateful to her PhD advisor Christopher Leininger for his patience and generosity.

\section{Preliminaries}

\noindent Throughout the remainder of the paper we will assume that $S$ is a closed surface of genus at least $2$.

\subsection*{Non-positively curved Euclidean cone metrics}

A geodesic metric $\varphi$ on $S$ is a {\bf Euclidean cone metric} if it satisfies the following conditions. 

\begin{itemize}
    \item[(i)] Except at a finite set of points, which we will denote $\text{cone}(\varphi)$, $S$ is Euclidean. That is, on $S \setminus \text{cone}(\varphi)$, $\varphi$ is locally isometric to $\mathbb R^2$ with the Euclidean metric. 
    
    \item[(ii)] At every point $x \in \text{cone}(\varphi)$ there is an $\varepsilon$-neighborhood of $x$ isometric to some cone, where a cone is obtained by gluing together a finite number of sectors of $\varepsilon$-balls about $0$ in $\mathbb R^2$ by isometries.
\end{itemize}

We will call the points $x \in \text{cone}(\varphi)$ {\bf cone points}. Note that condition (ii) gives us a well defined notion of {\bf cone angle} about each cone point $x$, denoted $\text{c}(x)$, as the sum of the angles in the sectors mentioned above. For $x \in S \setminus \text{cone}(\varphi)$ we will define $c(x) = 2 \pi$.

Recall the following well known fact relating cone angles and Euler characteristic which we will need in the proof of Lemma \ref{lemma:cylcurve}. The interested reader can find a proof of this special case of the Gauss-Bonnet formula in \cite{bankovic2014}.

\begin{prop}[Gauss--Bonnet formula] \label{prop:gauss-bonnet} Let $R$ be a surface of genus $g \geq 0$ together with a Euclidean cone metric $\varphi$. Then \[2 \pi \chi(R) = \sum_{x \in \text{cone}(\varphi)} (2 \pi - c(x)).\] \end{prop}

A Euclidean cone metric $\varphi$ is said to be {\bf non-positively curved} if it is locally $\CAT(0)$. By Gromov's Link Condition non-positive curvature of a Euclidean cone metric $\varphi$ is equivalent to $c(x) \geq 2 \pi$ for every $x \in S$, see \cite{bridhaef99}. Denote the set of non-positively curved Euclidean cone metric by $\text{Flat}(S)$. 

\subsection*{Flat metrics and their geodesics}

Since $\varphi$ has zero curvature away from cone points, then associated to $\varphi$ is a well-defined {\bf holonomy homomorphism} \[P_{\varphi}: \pi_1(S \setminus \text{cone}(\varphi), x) \to \text{SO}(2),\] which is given by parallel transport around loops based at $x$. The image under $P_{\varphi}$ of $\pi_1(S \setminus \text{cone}(\varphi),x)$ will often be referred to simply as the {\bf holonomy} of $\varphi$ and will be denoted by $\text{Hol}(\varphi)$. 

Let $\text{Flat}_q(S) = \left\{ \varphi \in \text{Flat}(S) ~|~ \text{Hol}(\varphi) \leq \left< r_{\frac{2 \pi}{q}} \right>\right\},$ where $r_{\theta}$ is a rotation through the angle $\theta$. We can also think of $\text{Flat}_q(S)$ as the space of flat metrics coming from holomorphic $q$-differentials. For a discussion of this complex analytic viewpoint of $\text{Flat}_q(S)$ see \cite{bankovic2014}. 

For $\varphi \in \text{Flat}_q(S)$, we will define the {\bf holonomy almost trivializing} branched cover to be the branched cover corresponding to $P_{\varphi}^{-1}(\left<r_{\pi} \right>)$. Note that when $q$ is odd this is the same as the {\bf holonomy trivializing cover}, which is the branched cover corresponding to the kernel of $P_{\varphi}$. 

\begin{mydef} \label{def:q-simple} A curve $\alpha \in \mathcal{C}$ is a {\bf $q$-simple curve} if there exists $\varphi \in \text{Flat}_q(S)$ and a simple closed curve $\alpha'$ in the holonomy almost trivializing cover of $(S, \varphi)$ whose geodesic representative projects to the geodesic representative of $\alpha$.\end{mydef}

Let $\mathcal{S}_q$ be the set of $q$-simple curves for some $q \geq 1$. Note that $\mathcal{S}_q$ is a subset of the collection of curves $\alpha \in \mathcal{C}$, which are images of homotopy classes of simple closed curves in some cyclic $q'$-fold branched cover, where $q' | q$. 

Recall that a closed curve $\gamma$ on $S$ is a continuous map $\gamma: S^1 \to S$. However, we will often identify a curve with its image in $S$ and simply denote $\gamma: S^1 \to S$ by $\gamma \subset S$. Furthermore, we will often abuse notation and conflate a curve with its unoriented homotopy class.

An important feature of flat metrics is the structure of their geodesics. In particular, geodesics in a flat metric come in two distinct flavors as stated in the following proposition whose proof can be found in \cite{bankovic2014}. 

\begin{prop}[Proposition 2.2, \cite{bankovic2014}] \label{prop:bankovic}
Let $\varphi \in \text{Flat}(S)$. A closed curve $\gamma \subset S$ is a $\varphi$-geodesic if and only if $\gamma$ is either a closed Euclidean geodesic or $\gamma$ is a concatenation of Euclidean line segments between cone points such that the angles between consecutive segments are $\geq \pi$ on each side of the curve $\gamma$.
\end{prop}

For any $\varphi \in \text{Flat}(S)$, up to parametrization, $\varphi$-geodesics are unique except for when there is a family of parallel geodesics filling up a Euclidean cylinder (that is when they are of the first type in Proposition \ref{prop:bankovic}). Note that in general a Euclidean cylinder in $(S, \varphi)$ will not be embedded. We will call the homotopy class of a curve with non-unique geodesic representatives foliating a cylinder a {\bf cylinder curve} and we will define the {\bf cylinder set} of $\varphi$, denoted $\cyl(\varphi) \subset \mathcal{C}(S)$, to be the set of all cylinder curves with respect to $\varphi$. Note that cylinder curves for a metric in $\text{Flat}_q(S)$ are necessarily $q$-simple. We will call a geodesic segment between two (not necessarily distinct) cone points that has no cone points in its interior a {\bf saddle connection}.

Note that, unlike geodesic representatives of curves in a hyperbolic metric, $\varphi$-geodesic representatives of curves do not always intersect transversely. Indeed two $\varphi$-geodesics can agree along multiple saddle connections. However, given two curves $\alpha, \beta \subset (S, \varphi)$ we can make precise the meaning of a point of intersection $p \in \alpha \cap \beta$ by recalling that intersection points correspond bijectively with $\pi_1(S)$-orbits of linking pairs of lifts $\widetilde{\alpha}, \widetilde{\beta} \subset \widetilde{S}$. See the proof of Theorem 6.3.10 in \cite{martelli16} for a careful description of this bijection. 

\subsection*{Geodesic currents}

Consider the universal cover $p: \widetilde{S} \to S$. For any geodesic metric $m$ on S there is an induced metric $\widetilde{m}$ on $\widetilde{S}$. Now consider the Gromov boundary $S^1_{\infty}$ of $\widetilde{S}$ with respect to $\widetilde{m}$. $S^1_{\infty}$ compactifies $\widetilde{S}$ to a closed disk and the action of $\pi_1(S)$ on $\widetilde{S}$ extends by homeomorphism to an action on this disk. 

We can think of $S^1_{\infty}$ as being independent of the choice of geodesic metric since for any other geodesic metric the identity on $\widetilde{S}$ extends to a homeomorphism between the corresponding closed disks. We will consider the set of unordered pairs of distinct points in $\widetilde{S}$ \[\mathcal{G}(\widetilde{S}) = (S^1_{\infty} \times S^1_{\infty} \setminus D) /_{(x,y) \sim (y,x)},\] where $D = \{(x, x) | x \in S^1_{\infty}\}$ is the diagonal. Note that the action of $\pi_1(S)$ on $\widetilde{S}$ gives an action on $\mathcal{G}(\widetilde{S})$.

We define a {\bf geodesic current} to be a $\pi_1(S)$-invariant Radon measure on $\mathcal{G}(\widetilde{S})$. The space of geodesic currents on $S$ equipped with the weak* topology will be denoted $\text{Curr}(S)$.

There is a natural way to associate a geodesic current to an essential closed curve $\gamma \subset S$. Consider the endpoints on $S_{\infty}^1$ of complete lifts $\widetilde{\gamma}:\mathbb R \to \widetilde{S}$ of $\gamma$ to $\widetilde{S}$, which form a discrete, $\pi_1(S)$-invariant set of points in $\mathcal{G}(\widetilde{S})$. The counting measure on this set defines a geodesic current on $S$ which we also denote $\gamma$. An important result about geodesic currents is the following theorem of Bonahon, see \cite{bonahon88}.

\begin{thm} \label{thm:bonahon} There exists a continuous, symmetric, bilinear form \[i: 
\text{Curr}(S) \times \text{Curr}(S) \to \mathbb R\] such that for any two currents $\gamma_1$, $\gamma_2$ associated to closed curves of the same name, $i(\gamma_1, \gamma_2)$ is the geometric intersection number of the homotopy classes of these curves.\end{thm}

Another important fact about geodesic currents was proved by Otal \cite{otal90}.

\begin{thm} Two currents $\mu_1, \mu_2 \in \text{Curr}(S)$ are equal if an only if $i(\gamma, \mu_1) = i(\gamma, \mu_2)$ for every curve $\gamma \in \text{Curr}(S)$. \end{thm}

We will be concerned with a certain class of geodesic currents called the {\bf Liouville currents} $L_m$ which are associated to certain types of metric $m$. The Liouville current has the geometricity property that the intersection form, given in Theorem \ref{thm:bonahon}, recovers the lengths of curves, i.e. for every essential closed curve $\gamma \subset S$ \[i(\gamma, L_m) = \ell_m(\gamma).\]

The Liouville current exists for many types of metrics. It generalizes the Liouville measure on geodesics in the hyperbolic plane given by cross ratios. For example, Erlandsson \cite{erlandsson} defines a type of combinatorial Liouville current associated to a simple generating set for a surface group which realizes the (word) length function on the surface (with respect to the chosen generating set) as the intersection with this current. A more careful treatment of the Liouville current associated to a flat metric in \cite{duchleinrafi10} and \cite{banklein18}. 

We will mostly concern ourselves with the support of the Liouville current for a metric $\varphi \in \text{Flat}(S)$, since a key tool in the proof of the Main Theorem is the following result of Duchin--Erlandsson--Leininger--Sadanand.

\begin{thm}[Support Rigidity Theorem, \cite{currents18}] \label{thm:support-rigidity}
Suppose $\varphi_1, \varphi_2$ are two unit-area flat metrics whose Liouville currents have the same support, $\text{supp}(L_{\varphi_1}) = \text{supp}(L_{\varphi_2})$. Then $\varphi_1, \varphi_2$ differ by an affine deformation, up to isotopy. If either metric has holonomy of order greater than $2$ (i.e. is not induced by a quadratic differential), then equal support implies that $\varphi_1$ and $\varphi_2$ differ by an isometry, isotopic to the identity. \end{thm}

We will find the following description of the structure of $\text{supp}(L_{\varphi})$ for a flat metric $\varphi$ very useful. It was proved by Bankovic--Leininger in \cite{banklein18}.

\begin{prop}[Corollary 3.5, \cite{banklein18}] \label{prop:basic-geodesics} For any $\varphi \in \text{Flat}(S)$, the support of $L_{\varphi}$ consists of endpoints of the closure of the set of nonsingular geodesics.\end{prop}

Finally, we will use a result of Vorobets about directions of periodic geodesics to show that the endpoints of lifts of cylinder geodesics are dense in the support of the Liouville current for any $\text{Flat}_q(S)$ metric. 

\begin{thm}[Theorem 1.2(b), \cite{vorobets05}] \label{thm:vorobets} Let $\varphi \in \text{Flat}_1(S)$. For almost every $x \in (S, \varphi)$ directions of periodic geodesics passing through the point $x$ are dense in $S^1$.
\end{thm}

As a corollary to Theorem \ref{thm:vorobets} we have the following.

\begin{cor} \label{corollary:vorobets} Let $\varphi \in \text{Flat}_q(S)$. The endpoints of lifts of cylinder geodesics are dense in the support of the Liouville current $L_{\varphi}$. \end{cor}

\begin{proof}
Consider the holonomy trivializing cover $p: S' \to S$. By Theorem \ref{thm:vorobets}, the set of tangent vectors to cylinder geodesics in $S'$ are dense in $T^1(S')$. Note that this cover $p$ has derivative \[\text{d}p: T^1(S') \to T^1(S)\] that maps the dense set of tangent vectors to cylinder geodesics to a dense set of tangent vectors to cylinder geodesics. This is because $p$ maps cylinder geodesics in $S'$ to cylinder geodesics in $S$. Since $\text{supp}(L_{\varphi})$ consists of the endpoints of the closure of the set of nonsingular geodesics, and any nonsingular geodesic is a limit of lifts of cylinder geodesics, then the endpoints of lifts of cylinder geodesics are dense in $\text{supp}(L_{\varphi})$. \end{proof}

\section{Spectral Rigidity for $q$-simple Curves}

In this section we will prove the Main Theorem. The key is the following statement that cylinder curves for $\varphi \in \text{Flat}_q(S)$ are determined by $\varphi$-lengths of curves in $\mathcal{S}_q$. 

\begin{prop} \label{prop:cylcurve} If $\varphi_1, \varphi_2 \in \text{Flat}_q(S)$ and $\lambda_{\mathcal{S}_q}(\varphi_1) = \lambda_{\mathcal{S}_q}(\varphi_2)$, then $\cyl(\varphi_1) = \cyl(\varphi_2)$.\end{prop}

Assuming Proposition \ref{prop:cylcurve}, we are now ready to prove the Main Theorem. 

\begin{proof}[Proof of the Main Theorem]
Suppose $\lambda_{\mathcal{S}_q}(\varphi_1) = \lambda_{\mathcal{S}_q}(\varphi_2)$. By Proposition \ref{prop:cylcurve}, we have that $\cyl(\varphi_1) = \cyl(\varphi_2)$. By Corollary \ref{corollary:vorobets}, we know that the endpoints of cylinder geodesics for $\text{Flat}_q$ metrics are dense in the support of the Liouville current. Thus, $\text{supp}(L_{\varphi_1}) = \text{supp}(L_{\varphi_2})$. 

When $|\text{Hol}(\varphi_1)| >2 $, we can apply Theorem \ref{thm:support-rigidity} to conclude that $\varphi_1$ and $\varphi_2$ differ by an isometry, isotopic to the identity. That is, $\varphi_1$ is equal to $\varphi_2$ in $\text{Flat}_q(S)$. 

On the other hand, suppose $|\text{Hol}(\varphi_1)| \leq 2$. Theorem \ref{thm:support-rigidity} tells us that $\varphi_1$ and $\varphi_2$ differ by an affine deformation up to isotopy. Note that an affine deformation preserves holonomy. Thus, $|\text{Hol}(\varphi_1)| = |\text{Hol}(\varphi_2)|$ and hence, $\varphi_1, \varphi_2 \in \text{Flat}_2(S)$ and we can conclude that $\varphi_1 = \varphi_2$ by Theorem 1 of \cite{duchleinrafi10}\end{proof}

\subsection*{Characterization of Cylinder Curves} 

The proof of Proposition \ref{prop:cylcurve} will follow a similar strategy to that of Duchin--Leininger--Rafi in Section 3 of \cite{duchleinrafi10}. In particular, Duchin--Leininger--Rafi characterize cylinder curves by comparing the lengths of images of simple closed curves under Dehn twists about cylinder curves to the lengths of their images under Dehn twists about non-cylinder curves. Note that a Dehn twist is a way to combine two intersecting curves (one of which is simple) via surgery at every point of intersection. However, since we are not working in the context of simple closed curves we will need to find a different way to combine curves. We will use the notion of a product of two curves at specified points of intersection. 

Given two curves $\alpha, \beta \in \mathcal C$ such that $i (\alpha, \beta) \neq 0$, let $\mathcal P (\alpha, \beta)$ be the collection of all points of intersection and all pairs of points of intersection between $\alpha$ and $\beta$. Note that $P \in \mathcal P(\alpha, \beta)$ is either a point $\{p\}$ or a pair of points $\{p, r\}$. 

\begin{mydef} \label{def:product} Let $\alpha, \beta \in \mathcal C$ such that $i(\alpha, \beta) \neq 0$ and let $P \in \mathcal P(\alpha, \beta)$. There are two cases, either $P = \{p\}$ or $\{p, r\}$. In either case, orient $\alpha$ and $\beta$ so that the intersection point $p$ contributes $+1$ to the algebraic intersection number $\left<\alpha, \beta \right>$. If $P = \{p\}$, then we will define $D_{\alpha}(\beta, P)$ to be the usual product of based loops $\alpha$ and $\beta$ based at $p$. If $P = \{p, r\}$, then we will define $D_{\alpha}(\beta, P)$ as follows. Let $\beta_0$ be the segment of $\beta$ starting  at $p$ and ending at $r$ and let $\beta_1$ be the segment of $\beta$ starting at $r$ and ending at $p$. Then $D_{\alpha}(\beta, P)$ is the concatenation $\alpha \cdot \beta_0 \cdot \alpha^{\varepsilon} \cdot \beta_1$ (which is a loop based at $p$), where $\varepsilon$ is the sign of the intersection between $\alpha$ and $\beta$ at $r$. \end{mydef}

Note that there are two choices of orientation on $\alpha$ and $\beta$ which make $p$ positive. However, the resulting products only differ by reversing orientation. So this product is well-defined, independent of choice, as an unoriented curve. Furthermore, if $\alpha$ is simple and $i(\alpha, \beta) = 1$ or $2$, then $D_{\alpha}(\beta, P)$, where $P = \{ \alpha \cap \beta \}$, is simply the Dehn twist of $\beta$ about $\alpha$. Note that we can also consider the $n$-th power of this product $D_{\alpha}^n(\beta, P)$ for any $P \in \mathcal P(\alpha, \beta)$, which is simply the product $D_{\alpha^n}(\beta, P)$. When the choice of $P \in \mathcal P(\alpha, \beta)$ is clear we will simple denote the product $D_{\alpha}^n(\beta, P)$ by $D_{\alpha}^n(\beta)$. 

We can now characterize whether a $q$-simple curve is a cylinder curve for some $\text{Flat}_q(S)$ metric by examining the length of its product with powers of $q$-simple curves. In particular, we will prove the following.

\begin{prop} \label{prop:character} For any $\alpha \in \mathcal{S}_q$ and any $\varphi \in \text{Flat}_q(S)$, we have $\alpha \in \cyl(\varphi)$ if and only if there exists $\beta \in \mathcal{S}_q$, $P \in \mathcal P(\alpha, \beta)$, and $N > 1$ such that for all $n \geq N$, we have $D_{\alpha}^n(\beta, P) \in \mathcal{S}_q$ and the following condition holds: \[ \ell_{\varphi}(D_{\alpha}^n(\beta, P)) < \ell_{\varphi}(D_{\alpha}^{n-1}(\beta, P)) + \ell_{\varphi}(\alpha) \cdot  |P|.\]\end{prop}

Note that Proposition \ref{prop:character} characterizes cylinder curves in terms of lengths of curves in $\mathcal{S}_q$. Thus, Proposition \ref{prop:cylcurve} follows as an immediate corollary. In order to complete the proof of the Main Theorem, we are only left to prove Proposition \ref{prop:character}. The forward implication will be proven in Lemma \ref{lemma:cylcurve} and the reverse implication in Lemma \ref{lemma:notcylcurve}.

\begin{lemma} \label{lemma:cylcurve} Let $\varphi \in \text{Flat}_q(S)$. If $\alpha \in \cyl(\varphi)$, then there exists $\beta \in \mathcal{S}_q$ with $i(\alpha, \beta) \neq 0$ and there exists $P \in \mathcal P(\alpha, \beta)$, such that for all $n \geq 1$, \[D_{\alpha}^n(\beta, P) \in \mathcal{S}_q \text{ and } \ell_{\varphi}(D_{\alpha}^n(\beta, P) < \ell_{\varphi}(D_{\alpha}^{n-1}(\beta, P)) + \ell_{\varphi}(\alpha)\cdot |P|.\] \end{lemma}

Before beginning the proof of Lemma \ref{lemma:cylcurve} we will give a brief preview of the main strategy. Fix a flat metric $\varphi \in \text{Flat}_q(S)$. Given a cylinder curve $\alpha$ on $(S, \varphi)$ we will construct a $q$-simple curve $\beta$ such that $D_{\alpha}^n(\beta)$ is $q$-simple for all $n$ and so that $D_{\alpha}^n(\beta)$ satisfies the necessary inequality on its length. To construct $\beta$ we will pass to the holonomy almost trivializing cover $f: (S' \varphi') \to (S, \varphi)$ where we will use a cylinder curve $\alpha'$, which maps to $\alpha$ under $f$, to construct a closed geodesic $\beta'$, so that $\beta'$ along with the product of $\beta'$ with powers of $\alpha'$ are mapped by $f$ to closed geodesics on $S$. We then use the fact that cylinder curves live in cylinders with Euclidean geometry, and hence, geodesics cannot make ``sharp turns" inside of cylinders, to finish the proof. 

\begin{proof} Let $\varphi \in \text{Flat}_q(S)$. There are two cases, depending on the parity of $q$. This is a reflection of our definition of $q$-simple curves. Recall that when $q$ is odd the holonomy almost trivializing branched cover coincides with the holonomy trivializing cover. Thus, it is illuminating to focus on the slightly simplified situation for odd $q$ first before mentioning the small complication introduced when $q$ is even (which we will do at the end of the proof). 

Suppose $q$ is odd. Since $\alpha \in \cyl(\varphi)$, we have a maximal isometrically immersed cylinder $A \subset S$ with core curve $\alpha$. Consider the holonomy trivializing branched cover $f: (S', \varphi') \to (S, \varphi)$ branched over the cone points. Throughout the proof we will use prime notation to denote when we are in the holonomy trivializing branched cover. Let $A' \subset X$ be a cylinder mapping to $A$ by $f$ with core curve $\alpha'$. We will use $A'$ to construct $\beta$. The proof divides into two more subcases depending on if and how $A'$ intersects itself. Note the interior of $A'$ is always embedded, but $A'$ may not be. 

\medskip

\noindent {\bf Case 1:} {\it $A'$ is embedded or meets itself in a finite number of cone points .} We will construct a simple closed curve $\beta' \subset S'$ whose geodesic representative projects to a geodesic in $S$ and which intersects $\alpha'$ exactly once. 

Consider a cone point $c_0$ on one boundary component of $A'$. For each $k > 0$, choose a dense ray emanating from $c_0$ making angle less than $\frac{\pi}{k}$ with the boundary of $A'$, which meets no other cone points. We will denote the point where this ray returns to the other boundary component of $A'$ by $c_k^0$ and we will denote the arc of this ray between $c_0$ and $c_k^0$ by $\delta_k'$. Note that $c_k^0$ is not a cone point because of our choice of ray. Let $c_k$ be the first cone point to the right of $c_k^0$. We will denote the segment from $c_k^0$ to $c_k$ along the boundary of $A'$ by $\sigma_k'$. This is illustrated in Figure \ref{fig:segmentemb}. 

\begin{figure}[!htbp]
\centering
\def\svgwidth{3in}
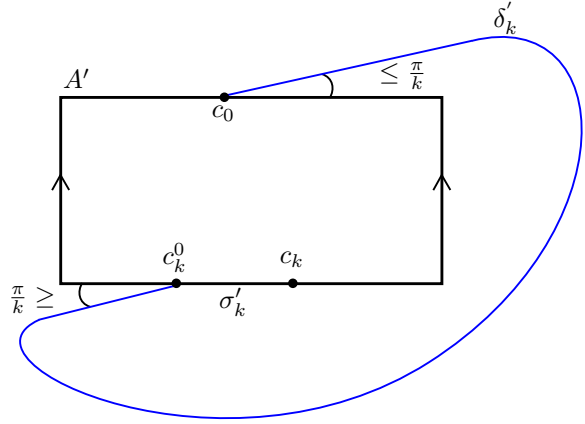
\caption{Segment $\delta_k' \cdot \sigma_k'$ from $c_0$ to $c_n$}
\label{fig:segmentemb}
\end{figure}

Now let $\gamma_k'$ be the geodesic representative of the homotopy class of $\delta_k' \cdot \sigma_k'$ rel. endpoints. Lift $\gamma_k'$, $\delta_k'$, and $\sigma_k'$ to the universal cover and consider the triangle they form. Denote the angles of the triangle by $\alpha_1,$ $\alpha_2,$ and $\alpha_3$ as shown in Figure \ref{fig:triangle}. By choosing $\delta_k'$ to be part of a ray with no singular points other than $c_0$ we ensured that this triangle does not degenerate.

\begin{figure}[!htbp]
\centering
\def\svgwidth{4in}
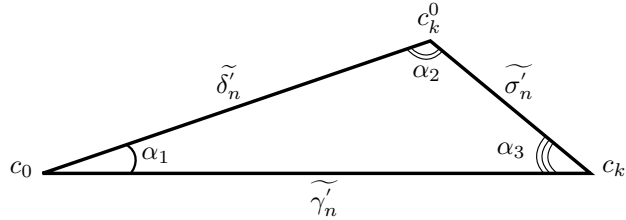
\caption{Triangle formed by the lifts of $\gamma_k'$, $\delta_k'$, and $\sigma_k'$}
\label{fig:triangle}
\end{figure}

Now consider the double of the triangle formed by the lifts $\widetilde{\gamma_k}, \widetilde{\delta_k}$, and $\widetilde{\sigma_k}$, which is a sphere, call it $R$, with the induced Euclidean cone metric. Note that the Euler characteristic of $R$ is $2$ and so by the Gauss--Bonnet formula given in Proposition \ref{prop:gauss-bonnet} we have \[-4 \pi = \sum_{x \in X} (c(x) - 2\pi),\] where we let $X$ be the set of cone points on $R$ with its induced Euclidean cone metric. Three of these cone points come from the vertices of the triangle and have cone angle $2 \alpha_i$, $i = 1, 2, 3$. The rest of the points in $X$ come from cone points, $x_1, \ldots x_j$ along $\widetilde{\gamma_k'}$ and have cone angles $> 2\pi$. (The fact that there are no cone points in the interior of the triangle follows from another application of Gauss--Bonnet.) So we have the following equality \[-4 \pi = 2(\alpha_1 + \alpha_3) - 4\pi + (2 \alpha_2 - 2 \pi) + \sum_{i = 1}^{j} (c(x_i) - 2 \pi).\] Rearranging we obtain \[0 = 2(\alpha_1 + \alpha_3) + (2 \alpha_2 - 2 \pi) + \sum_{i = 1}^j (c(x_i) - 2 \pi).\] 

Note that as $k \to \infty$, $\alpha_2 \to \pi$. Thus, we have that $\alpha_1, \alpha_3 \to 0$ as $k \to \infty$. Furthermore, $c(x_i) \to 2 \pi$ for all $1 \leq i \leq j$. 

Now choose $k > q$ sufficiently large so that $\pi + \frac{\pi}{q} \geq \frac{c(x_i)}{2} \geq \pi$ for each $1 \leq i \leq j$.  Connect $c_0$ to $c_k$ by a geodesic arc $d_k'$ inside of $A'$ that leaves $c_0$ at an angle of at most $\frac{\pi}{k}$. Since $k > q$, we have that $\gamma_k' \cdot d_k'$ makes an angle of at least $\pi$ on both sides of the cone points $c_0$ and $c_k$ and less than $\pi + \frac{\pi}{q}$ on one side of each of $c_0$ and $c_k$. In particular, $\gamma_k' \cdot d_k'$ is a geodesic, since $\gamma_k'$ is a geodesic and we have shown that at $c_0$ and $c_n$ it has angle at least $\pi$ on each side. 

Furthermore, the cone points $x_i$ together with $c_0$ and $c_k$ are the only cone points on $\gamma_k' \cdot d_k'$. So the choice of sufficiently large $k$ such that $\pi + \frac{\pi}{q} \geq \frac{c(x_i)}{2} \geq \pi$ together with the argument in the previous paragraph, allows us to conclude that at every cone point it meets $\gamma_k' \cdot d_k'$ makes angle between $\pi$ and $\pi +\frac{\pi}{q}$ on one side.

Let $\beta' = \gamma_k' \cdot d_k'$. Note that $\beta'$ is the geodesic representative of a simple closed curve, since $\beta'$ is homotopic to $\delta_k' \cdot \sigma_k' \cdot d_k'$. Consider the image of $\beta'$ under $f$ which we will denote by $\beta$. We will show that $\beta$ is indeed a geodesic on $S$. To do so we will use the fact that at every cone point it meets $\beta'$ makes angle at least $\pi$ and at most $\pi +\frac{\pi}{q}$ on one side. Since cone angles on $S$ have angle $\frac{2 \pi j}{q}$ for some $j > q$, then $\beta$ will also make angle at least $\pi$ and at most $\pi + \frac{\pi}{q}$ on one side of each cone point it meets, hence the cone angle on the other side is also at least $\pi$. Thus, $\beta$ is geodesic, as desired. 

\begin{figure}[!htbp]
\centering
\def\svgwidth{3in}
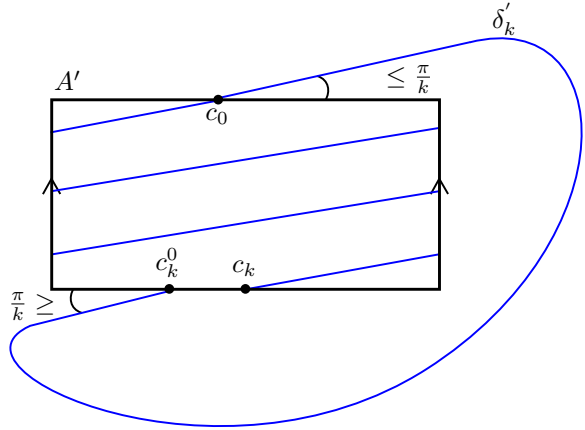
\caption{Geodesic arc in $A'$ connecting $c_0$ to $c_k$}
\label{fig:segmentembdiag}
\end{figure}

Since $\beta'$ nontrivially intersects with $\alpha'$, then $\beta$ nontrivially intersects $\alpha$. In fact, let $p'$ be the point of intersection between $\beta'$ and $\alpha'$. Then $p = f(p')$ will be the designated point of intersection and we will consider the product $D_{\alpha}^n(\beta, P)$ for $P = \{p\} \in \mathcal P(\alpha, \beta)$. 

Note that the geodesic representative of $D_{\alpha}^{n-1}(\beta)$ is obtained by replacing $d_k'$ in the construction of $\beta'$ (and hence, $\beta$) with an arc traversing the length of $A'$ exactly $n-1$ more times than $d_k'$ did. In particular, $D_{\alpha}^{n-1}(\beta) \in \mathcal{S}_q$. Furthermore, the geodesic representative of $D_{\alpha}^{n-1}(\beta)$ intersects $A$ in a geodesic arc with endpoints on the boundary of $A$. So we can construct a representative of $D_{\alpha}^n(\beta)$ by concatenating the geodesic representative of $D_{\alpha}^{n-1}(\beta)$ with the geodesic representative of $\alpha$ at $f(p')$.  The constructed representative of $D_{\alpha}^n(\beta)$, illustrated in Figure \ref{fig:alphabeta}, has length exactly $\ell_{\varphi}(D_{\alpha}^{n-1}(\beta)) + \ell_{\varphi}(\alpha)$. Clearly this representative is not geodesic since the angle formed inside of $A$ by $\alpha$ and $D_{\alpha}^{n-1}(\beta)$ at $p$ is strictly less than $\pi$. Thus, $\ell_{\varphi}(D_{\alpha}^n(\beta)) < \ell_{\varphi}(D_{\alpha}^{n-1}(\beta)) + \ell_{\varphi}(\alpha)$.

\begin{figure}[!htbp]
\centering
\def\svgwidth{3in}
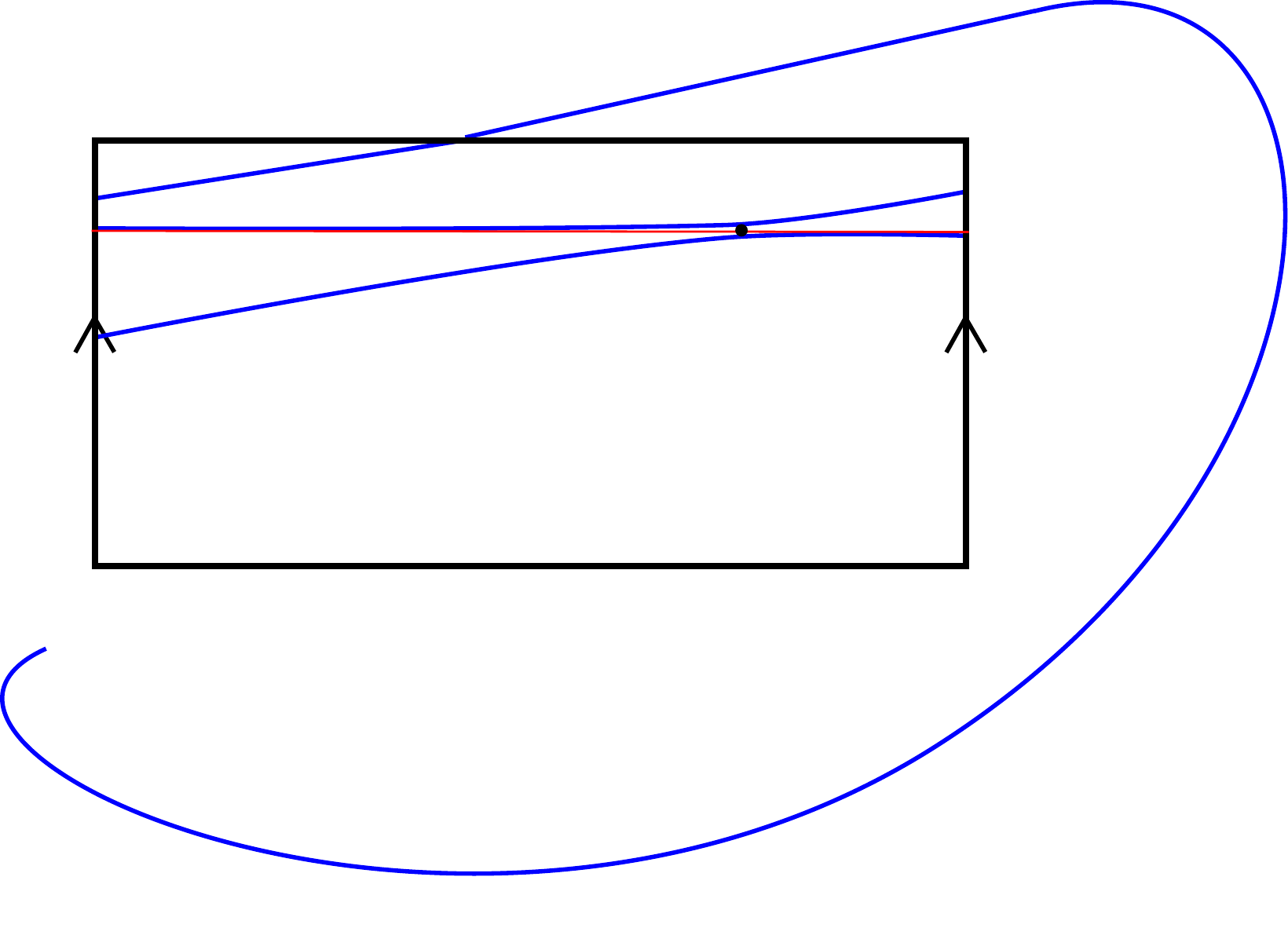
\caption{A representative of $D_{\alpha}^{n}(\beta)$ in blue and $\alpha$ in red}
\label{fig:alphabeta}
\end{figure}

\medskip

\noindent {\bf Case 2:} {\it $A'$ is not embedded and meets itself in at least one full saddle connection.} Since $A'$ meets itself in one full saddle connection, then we can build another cylinder curve $\beta'$ by taking a geodesic arc in $A'$ which begins and ends at the same point on one of the shared saddle connections on the boundary of $A'$. This new curve $\beta'$ is illustrated in Figure \ref{fig:immersed}. Now let $\beta = f(\beta')$. As above, since $\beta'$ nontrivially intersects with $\alpha'$ in some point $p'$, then $\beta$ nontrivially intersects $\alpha$ at $p = f(p')$ and we can consider the product $D_{\alpha}^n(\beta, P)$ for $P = \{p\} \in \mathcal P(\alpha, \beta)$. 

Note that $\beta$ is a cylinder curve in $S$. In fact, the cylinder for $\beta$ is built from a parallelogram contained in $A$, since $\beta$ is a geodesic arc in $A$ with endpoints that meet on a shared saddle connection on the boundaries of $A$, see Figure \ref{fig:immersed}. The geodesic representative of $D_{\alpha}^{n-1}(\beta)$ is obtained in the same way as in Case 1, namely it is obtained by replacing the geodesic arc in $A'$ used to construct $\beta'$ by a geodesic arc that traverses the length of $A'$ exactly $n-1$ more times than the original arc. In particular, $D_{\alpha}^{n-1}(\beta) \in \mathcal{S}_q$ and is also a cylinder curve. Thus, we can construct a representative of $D_{\alpha}^n(\beta)$, whose length is exactly $\ell_{\varphi}(D_{\alpha}^{n-1}(\beta)) + \ell_{\varphi}(\alpha)$, by concatenating $D_{\alpha}^{n-1}(\beta)$ and $\alpha$ at $f(p')$. Just as in Case 1, an angle argument tells us that this representative is not geodesic and thus, $\ell_{\varphi}(D_{\alpha}^n(\beta)) < \ell_{\varphi}(D_{\alpha}^{n-1}(\beta)) + \ell_{\varphi}(\alpha)$.

\begin{figure}[!htbp]
\centering
\def\svgwidth{2.5in}
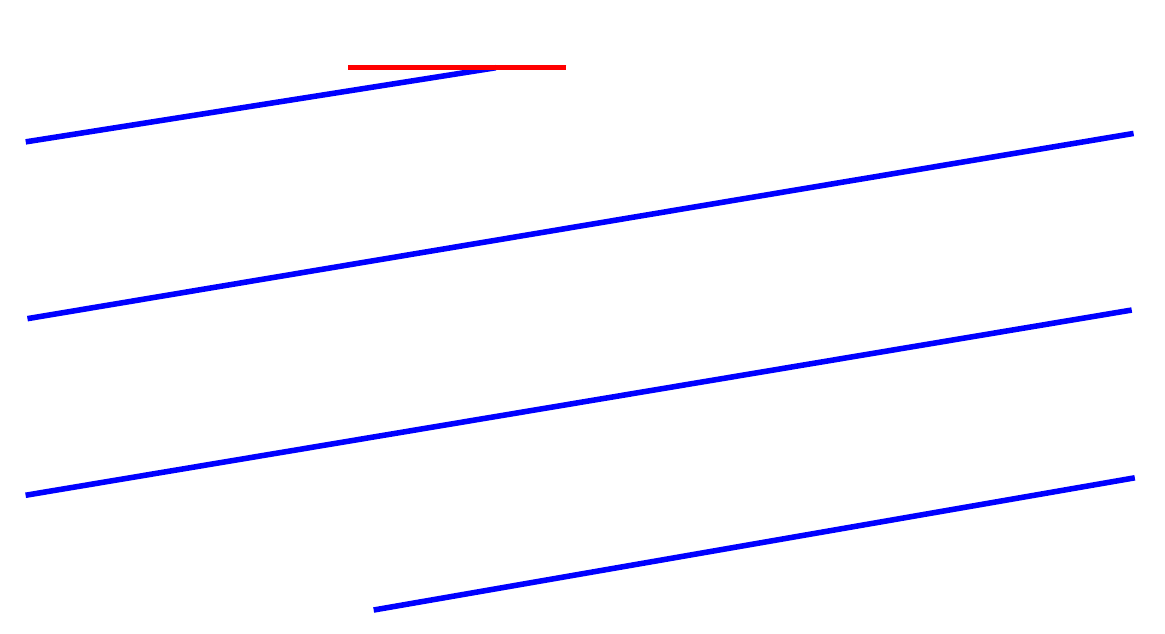
\caption{Constructing $\beta'$ when $A'$ meets itself in a saddle connection}
\label{fig:immersed}
\end{figure}

Now suppose that $q$ is even. We can carry out a nearly identical construction as we did for $q$ odd, only this time we use the holonomy almost trivializing branched cover. In this case, $\beta'$ may meet $\alpha'$ in two points $p'$ and $r'$, if $\delta_n'$ enters and exits in a single boundary component of $A'$. This situation is illustrated in Figure \ref{fig:qeven}. In this situation, we consider the product of $\beta = f(\beta')$ and $\alpha$ at $P = \{f(p'), f(r')\}$ and show that $D_{\alpha}^n(\beta) \in \mathcal S_q$ and $\ell_{\varphi}(D_{\alpha}^n(\beta)) < \ell_{\varphi}(D_{\alpha}^{n-1}(\beta)) + \ell_{\varphi}(\alpha) \cdot |P|$, as desired.\end{proof}

\begin{figure}[!htbp]
\centering
\def\svgwidth{2.5in}
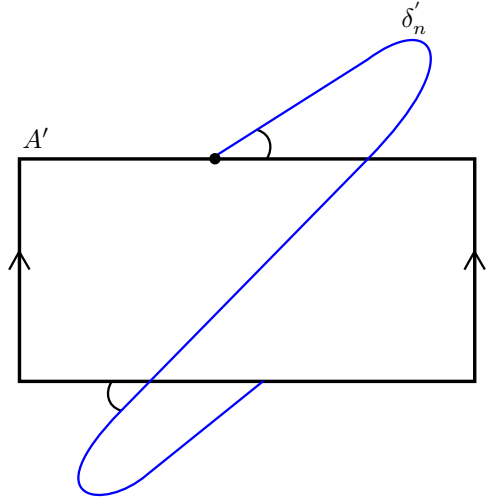
\caption{Constructing $\delta$ when $q$ is even}
\label{fig:qeven}
\end{figure}

Thus, we have proved the forward direction of Proposition \ref{prop:character} and it only remains to prove the reverse direction, which we will do next.

\begin{lemma} \label{lemma:notcylcurve} If $\alpha \not \in \cyl(\varphi)$, then for all closed curves $\beta$ with $i(\alpha, \beta) \neq 0$ there exists $N > 1$ such that for all $n \geq N$ and for all $P \in \mathcal P(\alpha, \beta)$, then \[\ell_{\varphi}(D_{\alpha}^n(\beta, P)) = \ell_{\varphi}(D_{\alpha}^{n-1}(\beta, P)) + \ell_{\varphi}(\alpha) \cdot |P|.\] \end{lemma}

The idea of the proof is to show that for $n$ sufficiently large, $D_{\alpha}^{n-1}(\beta)$ shares a full saddle connection with $\alpha$ so that we can obtain the geodesic representative of $D_{\alpha}^n(\beta)$ from $D_{\alpha}^{n-1}(\beta)$ by surgering in a copy of $\alpha$. 

%In order, to prove that $D_{\alpha}^{n-1}(\beta)$ shares a saddle connection with $\alpha$ we pass to the universal cover and view $D_{\alpha}^{n-1}(\beta)$ as a loop based at $p$ and then show that its geodesic axis has endpoints on the circle at infinity which approach the endpoints of the geodesic axis of $\alpha$ as $n$ grows. 

\begin{proof} Suppose $\alpha \not \in \cyl(\varphi)$ and let $\beta$ be a closed curve on $S$ such that $i(\alpha, \beta) \neq 0$. Fix $ P \in \mathcal P(\alpha, \beta)$. There are two cases: either $|P| = 1$ or $|P| = 2$. We will begin by assuming that $|P| = 1$. 

Throughout the proof we will let $\alpha$ and $\beta$ be the geodesic representatives of their free homotopy classes. Since $|P| = 1$, then $P = \{ p \}$ for some $p \in \alpha \cap \beta$. View $\alpha$ and $\beta$ as loops based at $p$. Note that they are the geodesic representatives of their based homotopy classes.

Fix a lift $\widetilde{p}$ of $p$ to the universal cover. This defines an action of $\pi_1(S, p)$ on $\widetilde{S}$. The axis, $\widetilde{\alpha}$, of $\alpha$ passes through $\widetilde{p}$ and projects to the based loop $\alpha$. Denote the endpoints of $\widetilde{\alpha}$ by $a^{\pm}$. Let $H^{\pm} \subset \widetilde{S}$ and $S^{\pm} \subset S_{\infty}^1$ be the two halfspaces and two subarcs into which $\widetilde{\alpha}$ divides the universal cover and its boundary circle, as shown in Figure \ref{fig:axisOfAlpha}. 

\begin{figure}[!htbp]
\centering
\def\svgwidth{2.5in}
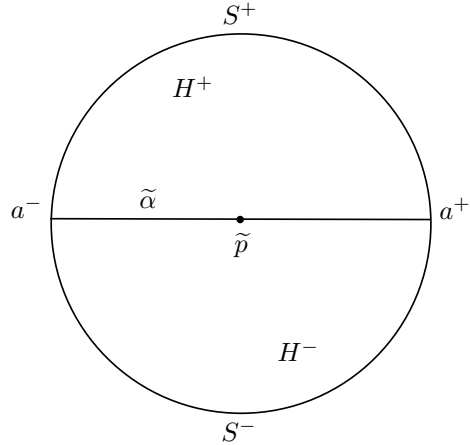
\caption{Schematic illustration of the axis of $\alpha$ in $\widetilde{S}$}
\label{fig:axisOfAlpha}
\end{figure}

Consider $\alpha^n \beta \alpha^{n'} \in \pi_1(S, p)$ for $n, n' > 0$. 

\bigskip

{\noindent \bf Claim:} For all disjoint neighborhoods $U^{\pm}$ of $a^{\pm}$, there exists $N, N' > 0$ such that for all $n \geq N$ and $n' \geq N'$ the axis, $\widetilde{\alpha^{n} \beta \alpha^{n'}}$, of $\alpha^{n} \beta \alpha^{n'}$ has end points in $U^{\pm}$. 

\begin{proof}[Proof of Claim] It is well known that hyperbolic isometries of $\widetilde{S} \cup S_{\infty}^1$ exhibit ``North-South dynamics". Namely, if $\psi$ is a hyperbolic isometry of $\widetilde{S} \cup S_{\infty}^1$ with attracting and repelling fixed points $x^+$ and $x^-$, respectively, then given disjoint neighborhoods $V^+, V^- \subset \widetilde{S} \cup S_{\infty}^1$ of $x^+, x^-$, respectively, there exists an $N \geq 1$ such that for all $n \geq N$: \[\psi^n((\widetilde{S} \cup S_{\infty}^1) \setminus V^-) \subset V^+ \text{ and } \psi^{-n}((\widetilde{S} \cup S_{\infty}^1) \setminus V^+) \subset V^-.\] This fact can be found in \cite{gromov87}. 

Recall that $\alpha$ acts as a hyperbolic isometry on $\widetilde{S}$ with attracting and repelling fixed points, $a^+$ and $a^-$, respectively. Let $U^{\pm} \subset S_{\infty}^1$ be disjoint neighborhoods of $a^{\pm}$. A sufficiently high power $n'$ of $\alpha$ will send $S_{\infty}^1 \setminus U^-$ into $U^+$. Note that, for $U^{\pm}$ sufficiently small, $\beta(\alpha^{n'}(U^+)) \subset S_{\infty}^1 \setminus U^-$, since the fixed points of $\beta$ are disjoint from those of $\beta$. Thus, a sufficiently high power $n$ of $\alpha$ will send $\beta \alpha^{n'}(U^+)$ into $U^+$. So we have that $\alpha^n \beta \alpha^{n'}(U^+) \subset U^+$. Hence, $\alpha^n \beta \alpha^{n'}$ has its attracting fixed point in $U^+$. A symmetric argument shows that $\alpha^n \beta \alpha^{n'}$ has its repelling fixed point in $U^-$ and we are done. \end{proof}

Since $\alpha$ is not a cylinder curve, then $\widetilde{\alpha}$ is a concatenation of saddle connections. If we consider the angles made on each side of the singularities, then we see that the angle cannot always be $\pi$ on either side of $\widetilde{\alpha}$, otherwise there is a nonsingular geodesic on $S$ parallel to $\alpha$ and $\alpha \in \cyl(\varphi)$, a contradiction. Thus, there exists a singularity $x_0^+$ on $\widetilde{\alpha}$ such that the angle at $x_0^+$ on the $H^+$ side made by the saddle connections meeting there is strictly greater than $\pi$, and likewise there is a singularity $x_0^-$  on $\widetilde{\alpha}$ chosen with respect to $H^-$. Since $\widetilde{\alpha}$ is periodic, then we can choose another pair of singularities $x_1^{\pm}$ with the same properties as $x_0^{\pm}$. In addition, we will choose $x_0^{\pm}$ and $x_1^{\pm}$ so that $x_0^-$ and $x_1^-$ are closer to $\widetilde{p}$ than $x_0^+$ and $x_1^+$. We will also choose $x_0^{\pm}$ and $x_1^{\pm}$ far enough apart so that $x_0^-$ together with $\widetilde{p}$ contains an entire fundamental domain for the action of $\alpha$, and likewise for $x_1^-$ together with $\widetilde{p}$. Now fix arbitrary geodesic rays $\gamma_{i}^{\pm}$ contained in $H^{\pm}$ starting at $x_i^{\pm}$ and making angle exactly $\pi$ with $\widetilde{\alpha}$. Let $A^{\pm}$ be the subarcs of the circle at infinity bounded by the endpoints of $\gamma_0^{\pm}$ and $\gamma_1^{\pm}$. This is illustrated in Figure \ref{fig:geodConvex}.

\begin{figure}[!htbp]
\centering
\def\svgwidth{2.5in}
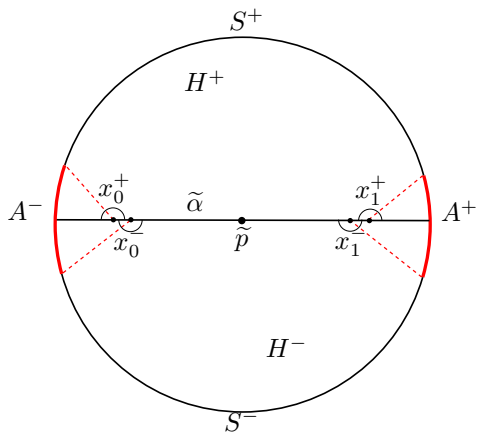
\caption{$\gamma_0^{\pm}$ and $\gamma_1{\pm}$ are shown by the dotted lines}
\label{fig:geodConvex}
\end{figure}

Choose $n$ and $n'$ sufficiently large so that one endpoint of the axis $\widetilde{\alpha^{n} \beta \alpha^{n'}}$ lies in $A^{+}$ and the other endpoint lies in $A^{-}$. Note that the axis $\widetilde{\alpha^{n} \beta \alpha^{n'}}$ must coincide with $\widetilde{\alpha}$ for at least the portion from $x_0^{-}$ to $x_1^{+}$ which includes $\widetilde{p}$. This is because the shaded region in Figure \ref{fig:shadedRegion} is geodesically convex, since it is the intersection of two half planes, and thus any geodesic with endpoints in $A^{\pm}$ must remain in this region. In particular, the geodesic representative of the based loop $\alpha^{n} \beta \alpha^{n'}$ starts and ends with a copy of the geodesic $\alpha$, and this is therefore also the geodesic representative of the free homotopy class. The geodesic representative of the based loop $\alpha^{n + 1} \beta \alpha^{n'}$ is obtained from that of $\alpha^{n} \beta \alpha^{n'}$ exactly by surgering in a copy of $\alpha$ and this is hence also the geodesic representative of its free homotopy class.

\begin{figure}[!htbp]
\centering
\def\svgwidth{2.5in}
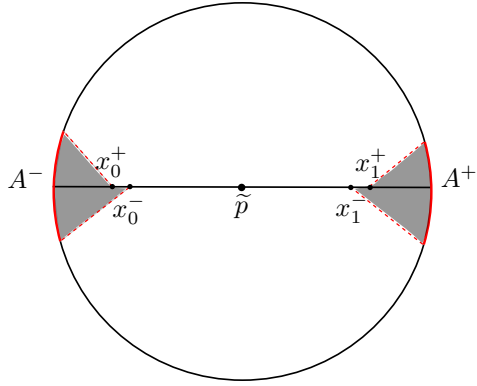
\caption{The shaded regions together with the connecting arc is geodesically convex}
\label{fig:shadedRegion}
\end{figure}

Note that $\beta \alpha^{n+n'}$ is conjugate to $\alpha^{n} \beta \alpha^{n'}$, and so the geodesic representatives of their free homotopy classes are the same. Thus, for all $n$ and $n'$ sufficiently large, we have that \begin{align*} \ell_{\varphi}(D_{\alpha}^{n+n'+1}(\beta)) = & \ell_{\varphi}(\beta \alpha^{n+n'+1})  \\ & = \ell_{\varphi}(\alpha^{n+1} \beta \alpha^{n'}) = \ell_{\varphi}(\alpha^{n} \beta \alpha^{n'})  + \ell_{\varphi}(\alpha) \\ & = \ell_{\varphi}(\beta \alpha^{n+n'}) + \ell(\alpha) = \ell_{\varphi}(D_{\alpha}^{n + n'}(\beta)) + \ell_{\varphi}(\alpha). \end{align*} 

Now suppose that $|P| = 2$. Thus, $P = \{p, r\}$ for some $p, r \in \alpha \cap \beta$. Recall that $D_{\alpha}^{n+n'}(\beta)$ is the concatenation $\alpha^{n+n'} \beta_0 \alpha^{\varepsilon \cdot (n + n')} \beta_1$ (which is a loop based at $p$) where $\beta_0$ and $\beta_1$ are the segments of $\beta$ defined in Definition \ref{def:product} and $\varepsilon$ is the sign of the intersection between $\alpha$ and $\beta$ at $r$.  Note that $\alpha^{n+n'} \beta_0 \alpha^{\varepsilon \cdot (n+n')} \beta_1$ is conjugate to $\alpha^{n} \beta_0 \alpha^{\varepsilon \cdot (n +n')} \beta_1 \alpha^{n'}$ and so the geodesic representatives of their free homotopy classes are the same. Furthermore, $\beta_0 \alpha^{\varepsilon \cdot (n + n')} \beta_1$ is also a loop based at $p$ that can be expressed as a product of loops $\beta_0' \alpha^{\varepsilon \cdot (n + n')} \beta_1'$ with $\beta_0', \beta_1' \in \pi_1(S, p)$. Since $\beta_0, \beta_1$ have fixed points disjoint from those of $\alpha$, then we can use essentially the same argument as in the Claim above, replacing $\beta$ with $\beta_0 \alpha^{\varepsilon \cdot (n + n')} \beta_1$, to show that for all disjoint neighborhoods $U^{\pm}$ of $a^{\pm}$, there exists $N, N' > 0$ such that for all $n \geq N$ and $n' \geq N'$ the axis of $\alpha^{n} \beta_0 \alpha^{\varepsilon \cdot (n+n')} \beta_1 \alpha^{n'}$ has endpoints in $U^{\pm}$. The rest of the proof then proceeds as above. \end{proof}

The completion of the proof of Lemma \ref{lemma:notcylcurve} completes the proof of Proposition \ref{prop:character} and hence, the proof of the Main Theorem.

\section{Open Questions}

Although we have already given a characterization of $\text{Flat}_q(S)$ metrics by a smaller subset of curves than was previously known, it would be interesting to know if $\mathcal{S}_q$ is essentially the smallest set of curves that will suffice.

For example, Duchin--Leininger--Rafi gave a complete characterization of which sets of simple closed curves on $S$ are spectrally rigid over $\text{Flat}(S)$ when $S$ has sufficient complexity. Let $\mathcal{PMF} = \mathcal{PMF}(S)$ denote Thurston's space of projective measured foliations on $S$.

\begin{thm} \label{thm:dlr2} If $3g - 3 + n \geq 2$, then $\Sigma \subset \mathcal{S} \subset \mathcal{PMF}$ is spectrally rigid over $\text{Flat}(S)$ if and only if $\Sigma$ is dense in $\mathcal{PMF}$. \end{thm}

In light of Theorem \ref{thm:dlr2}, we could ask whether there is a significantly smaller subset of $\mathcal{S}_q$ which is rigid over $\text{Flat}_q(S)$. Here we consider the closure, $\overline{\mathcal{S}_q}$, of $\mathcal{S}_q$ inside the set of projective currents on $S$, denoted $\mathcal{P}\text{Curr}(S)$.

\begin{conj} If $\Sigma \subset S_q$ is spectrally rigid over $\text{Flat}_q(S)$, then is $\overline{\Sigma} = \overline{\mathcal{S}_q}$? \end{conj}

In the same vein, we can ask what $\mathcal{S}_q$ looks like. 

\begin{conj} What is $\overline{\mathcal{S}_q} \subset \mathcal{P}\text{Curr}(S)$? Is there a ``nice" description of it? \end{conj}

%Note that there is an embedding of $\text{Flat}_2(S)$ into the space of geodesic currents given by sending a flat metric $\varphi \in \text{Flat}_2(S)$ to its associated Liouville current $L_{\varphi}$, see \cite{duchleinrafi10}. Furthermore this map remained an embedding after projectivizing. Duchin--Leininger--Rafi proved that the closure of $\text{Flat}_2(S)$ inside the space of projective currents is the space of projective mixed structures, which leads us to ask the following.
%
%\begin{conj} How does $\text{Flat}_q(S)$ embed into $\mathcal{P}\text{Curr}(S)$? What is the closure of $\text{Flat}_q(S)$ inside $\mathcal{P}\text{Curr}(S)$? Can we give a characterization of it in terms of some natural set of projective structures? \end{conj}

\bibliographystyle{plain}

\bibliography{qdifferentials}

\end{document}